\documentclass[12pt]{amsart}

\usepackage{amsmath, amsthm, amssymb}
\usepackage{tikz}
\usepackage{graphicx}
\usepackage{amsfonts}
\usepackage{pgfplots}
\usepackage{comment}
\usepackage{esint}
\usepackage{cancel}
\usepackage{eucal}
\usepackage{enumerate}
\usepackage{xcolor}
\usepackage[colorlinks,linkcolor=blue]{hyperref}

\allowdisplaybreaks[4]
\numberwithin{equation}{section}

\newtheorem{thm}{Theorem}[section]
\newcommand{\bt}{\begin{thm}}
\newcommand{\et}{\end{thm}}

\newtheorem{cor}[thm]{Corollary}   
\newcommand{\bc}{\begin{cor}}
\newcommand{\ec}{\end{cor}}

\newtheorem{lem}[thm]{Lemma}   
\newcommand{\bl}{\begin{lem}}
\newcommand{\el}{\end{lem}}

\newtheorem{prop}[thm]{Proposition}
\newcommand{\bp}{\begin{prop}}
\newcommand{\ep}{\end{prop}}

\newtheorem{defn}[thm]{Definition}
\newcommand{\bd}{\begin{defn}}    
\newcommand{\ed}{\end{defn}}

\newtheorem{rmrk}[thm]{Remark}   

\newcommand{\br}{\begin{rmrk}}
\newcommand{\er}{\end{rmrk}}

\newcommand{\thmref}[1]{Theorem~\ref{#1}}
\newcommand{\secref}[1]{Section~\ref{#1}}
\newcommand{\lemref}[1]{Lemma~\ref{#1}}
\newcommand{\defref}[1]{Definition~\ref{#1}}

\newcommand{\propref}[1]{Proposition~\ref{#1}}

\newcommand{\be}{\begin{equation}}
\newcommand{\ee}{\end{equation}}

\newcommand{\N}{\mathbb{N}}

\newcommand{\R}{\mathbb{R}}

\newcommand{\Z}{\mathbb{Z}}

\newcommand{\g}{\overline{g}}
\newcommand{\e}{\overline{e}}
\newcommand{\on}{\overline{\nabla}}
\newcommand{\pr}{\partial_r}

\newcommand{\vol}{{\rm vol}}
\newcommand{\Sc}{{\rm Sc}}
\newcommand{\Ric}{{\operatorname{Ric}}}

\begin{document}

\title[PMT for AF manifolds with isolated conical singularity]{Positive mass theorem for asymptotically flat manifolds with isolated conical singularities}

\author{Xianzhe Dai}
\address{
Department of Mathematics, 
University of Californai, Santa Barbara
CA93106, USA}
\email{dai@math.ucsb.edu}

\author{Yukai Sun}
\address{Key Laboratory of Pure and Applied Mathematics, 
School of Mathematical Sciences, Peking University, Beijing, 100871, P. R. China
}
\email{sunyukai@math.pku.edu.cn}

\author{Changliang Wang}
\address{
School of Mathematical Sciences and Institute for Advanced Study, Tongji University, Shanghai 200092, China}
\email{wangchl@tongji.edu.cn}

\date{}


\begin{abstract}
We prove the positive mass theorem for asymptotical flat (AF for short) manifolds with finitely many isolated conical singularities. 
We do not impose the spin condition. Instead we use the conformal blow up technique which dates back to Schoen's final resolution of 
the Yamabe conjecture.
\end{abstract}

\maketitle

\tableofcontents

\section{Introduction}

The famous Positive Mass Theorem states that an asymptotically flat (AF for short; also called asymptotically Euclidean) smooth manifold with nonnegative scalar curvature must have nonnegative ADM mass, and the mass is zero iff the manifold is the Euclidean space. This result was first proven by R. Schoen and S.-T. Yau \cite{SY-PMT} for manifolds of dimension between $3$ and $7$ using minimal surface method (for the recent progress of Schoen and Yau's argument in higher dimension, see \cite{SY1}). E. Witten \cite{Witten-PET} proved the Positive Mass Theorem for spin manifolds of arbitrary dimension by using the Dirac operator.  

It is a natural problem to generalize the Positive Mass Theorem to singular Riemannian manifolds (smooth manifolds endowed with non-smooth metrics), partially motived by the investigation of weak notions of nonnegative scalar curvature, and the study of stability aspect of positive mass theorem. There has been a lot of interesting works in this problem, see e.g. \cite{Grant-Tassotti, JSZ2022, Lee-PAMS-2013, Lee-LeFloch, Li-Mantoulidis, DG2012, Miao2002,Shi-Tam-PJM-2018}. These deal with metric singularities  occurring on smooth manifolds, and are usually assumed to be continuous and have $W^{1, p}$-regularity for $ p \geq n$. 

It was observed in \cite[Section 4.2]{Bray-Jauregui-AJM-2013} (see also Proposition 2.3 in \cite{Shi-Tam-PJM-2018}) that the negative mass Schwarzchild metric can be viewed as an AF metric with a $r^{\frac{2}{3}}$-horn singularity, which is scalar flat but has negative mass. A model $r^b$-horn metric is given as
\begin{equation}
\overline{g}_b := dr^2 + r^{2b} g^N, \ \ \text{for some} \ \ b >0,
\end{equation}
on a product manifold $(0, 1) \times N$, where $g^N$ is a Riemannian metric on closed manifold $N$, and near a $r^b$-horn singularity the metric is asymptotic to the model $r^b$-horn metric as $r \to 0$. In particular, for $b=1$, it is a conical singularity. These singular metrics are not continuous at the singular point, which is not even a topological manifold point if the cross section $N$ is not a sphere. In contrast to the negative mass Schwarzchild metric, the authors in \cite{DSW-spin-23} proved Positive Mass Theorem for AF spin manifolds with isolated $r^b$-horn singularities for $b \geq 1$, and in particular for AF spin manifolds with isolated conical singularities. In \cite{Li-Mantoulidis}, Li and Mantoulidis proved Positive Mass Theorem for metrics with conical singularities along a codimension two submanifold (aka edge metrics) on smooth AF manifolds. 
In \cite{TV2023}, Ju and Viaclovsky proved a positive mass theorem for AF manifolds with isolated orbifold singularities.

Conical singularities also occur naturally in the study of the near horizon geometry of black holes. Motivated from this consideration, a positive mass theorem for conical singularity under the nonnegative Ricci curvature is established in \cite{Lucietti-21} (modulo some analytic details). 

In this paper, we prove Positive Mass Theorem for general AF manifolds with isolated conical singularity, i.e., without imposing the spin condition.
\begin{thm}\label{thm-main}
Let $(M^n, g)$ be an AF manifold with finitely many isolated conical singularity and $n \geq 3$. If the scalar curvature is nonnegative on the smooth part, then the ADM mass $m(g)$ is nonnegative. Furthermore, the mass $m(g) = 0$ if and only if $(M^n, g)$ is isometric to $(\R^n, g_{\R^n})$.
\end{thm}

Our result also holds for multiple AF ends, as we indicate later.
In 3-dimension, by estimating capacity and Willmore functional of the cross sections, as an application of mass-capacity inequalities in \cite{Miao2023}, Miao \cite{Miao-mass-capacity-2024} obtained the nonnegativity of the mass of an AF manifold with $r^b$-horn singularity with $b>\frac{2}{3}$, provided that the relative homology group $H_2(M\setminus B_{r_0}(o), \partial B_{r_0}(o)) = 0$, here $o$ is the singular point and $B_{r_0}(o)$ is the geodesic ball centered at $o$ with radius $r_0$.

The precise definition of AF manifold with conical singularity is given in \secref{sec-conic}. These singular manifolds are not complete. This is one of main differences from the classical Positive Mass Theorem on complete AF manifolds.  Some positive mass theorems on incomplete AF manifolds have been established in \cite{LLU-CVPDE-23,LUY-JDG}, provided that the scalar curvature on the AF end is nonnegative and satisfies some quantitative positive lower bound on a bounded domain near the AF infinity.

The non-negativity of the mass in \thmref{thm-main} is proved in \secref{sec-nonnegativity}. The basic strategy is to construct an AF manifold with boundary, and then apply the results about the positive mass theorem for manifolds with boundary, e.g. in \cite{Herzlich-CMP-97, Herzlich-02, Hirsch-Miao-20, Miao2023} and others. A straightforward way to construct an AF manifold with boundary is to remove a conical neighborhood of the singular point from an AF manifold with a conical singularity. This idea works well for $r^b$-horn singularity with $ b > 1$, as we have done in \cite{DSW-spin-23}. However, in the case of $b < 1$, the mean curvature of the boundary with respect to the normal pointing to AF end would be too large (relative to a certain negative power of the area of the boundary) to apply positive mass theorems for manifolds with boundary. As mentioned above, the negative mass Schwarzchild metric ($b = \frac{2}{3}$) gives a counter-example to the positive mass theorem with $r^b$-horn singularity for $b <1$. In the case of conical singularity ($b=1$), this simple construction may not give desired AF manifold with boundary either, since the required upper bound for the mean curvature seems not easy to be estimated.

Instead, in the case of conical singularity, before cutting off a part of the manifold, we first apply a conformal blow up technique to the AF metric with a conical singularity, following Ju and Viaclovsky's approach for orbifold singularity in \cite{TV2023}.  More precisely, we first conformally deform the AF metric with conical singularity by using a harmonic function, which has asymptotic order like Green's function (i.e. $r^{2-n}$) near the conical singularity, and is asymptotic to $1$ near the AF infinity. One may think the harmonic function as the Green function with its pole at the conical singularity. This conformal deformation flips over a cone, converting the conical singularity to the infinite end of the cone, and keeps the AF end up to a higher order perturbation. Then we remove the conical end, and obtain an AF manifold with boundary whose mean curvature is negative, and so we can apply Hirsch and Miao's results about positive mass theorem with boundary \cite{Hirsch-Miao-20} to prove the non-negativity of the mass. 

Note that by conformally blowing up the conical singularity, we obtain a Riemannian manifold with two ends, one is an AF end, and another is an asymptotically conical end, i.e. this manifold is an AF manifold with an additional complete end. Therefore, we can also apply results about the positive mass theorems on AF manifolds with arbitrary ends in \cite{LUY-JDG, Zhu-IMRN-23} (for dimensions between 3 and 7) and \cite{cecchini2021} (for spin manifolds) to prove the nonnegativity of the mass in \thmref{thm-main}.

In a related work \cite{AB2003}, K. Akutagawa and B. Botvinnik studied the Yamabe problem for cylindrical manifolds, and obtained positive mass theorem for certain AF manifold with two specific conical singularities, which is constructed by a conformal blow up of a cylinder. The conformal blow up technique dates back to Schoen's solution of Yamabe problem \cite{Schoen-Yamabe-problem}, and has been a very useful and powerful tool in the study of various problems related to scalar curvature. The Schwarzschild metric $g=(1+\frac{m}{r^{n-2}})^{\frac{4}{n-2}}g_{\mathbb{R}^{n}}$ on $\mathbb{R}^{n}\backslash \{0\}$, for $m>0$, can be viewed as the conformal blowup of the origin by the Green's function $1+\frac{m}{r^{n-2}}$, which is asymptotic to $1$ as $r\to \infty$ and has the asymptotical order $r^{2-n}$ as $r\to 0$.    

The key in our argument is to find a suitable harmonic function. In contrast to the case of orbifold singularity, we need to solve for a harmonic function with certain prescribed asymptotic behavior on the AF manifold with conical singularity, which may not be lifted to a smooth manifold. To deal with the conical singularity, we employ weighted Sobolev spaces introduced in \cite{DSW-spin-23}, which have two weights, one for the conical singularity and the other for the AF infinity. 

The rigidity result in \thmref{thm-main} is proved in \secref{sec-rigidity}. We first prove that the mass $m(g) = 0$ and $\Sc_g \geq 0$ imply the AF metric with conical singularity must be Ricci flat, by deforming the metric in a compact domain away from conical singularity, and following a conformal change. In the case of orbifold singularity, as shown in \cite{TV2023}, by applying the Bishop-Gromov's volume inequality for orbifolds, the Euclidean volume growth at AF infinity then implies that the metric has to be flat. It is more delicate and involved in the case of conical singularity, since the Bishop-Gromov's volume inequality may not hold for AF manifolds with conical singularity.  We need to get the rigidity for the cross section of the cone so that we can get more information about the regularity of the metric. For that, we solve harmonic coordinates on AF manifolds with conical singularity, and study its asymptotic behavior near the conical singularity, which can be linked to the first nonzero eigenvalue on the cross section. Then we could apply Obata's rigidity theorem to obtain that the cross section must be isometric to the standard sphere. This then implies that the metric is continuous and has $W^{1, n}$-regularity, and so the rigidity result in \cite{JSZ2022} can be applied. 

The paper is organized as follows. In Section \ref{sec-conic}, we introduce asymptotically flat (AF) manifolds with isolated conical singularities and their ADM mass and collect some basic facts on cones. Section \ref{sect-analysis} is devoted to the analysis on AF manifolds with isolated conical singularities. We review the weighted Sobolev spaces introduced in \cite{DSW-spin-23} and discuss the corresponding elliptic estimates for the Laplace operators, their Fredholm properties, and surjectivity. This is used in Section \ref{sec-nonnegativity} and \ref{sec-rigidity} to obtain asymptotically harmonic function with certain prescribed asymptotic behavior at the conical singularity or AF infinity. In Section \ref{sec-nonnegativity} and \ref{sec-rigidity} we prove the positive mass theorem for conical singularity (Theorem \ref{thm-main}). 

{\em Acknowledgement}: The authors are grateful to Jeff Viaclovsky for suggesting that something like this could work. Xianzhe Dai is partially supported by the Simons Foundation. Yukai Sun is partially funded by the National Key R\&D Program of China Grant 2020YFA0712800.  Changliang Wang is partially supported by the Fundamental Research Funds for the Central Universities and Shanghai Pilot Program for Basic Research.

\section{AF manifolds with isolated conical singularities} \label{sec-conic}
In this section, we give the precise definition of what we call asymptotically flat (AF for short) manifolds with finitely many isolated conical singularities, and the definition of the mass (at infinity) for these singular manifolds. 
\begin{defn}\label{defn-conic-mfld}
{\rm
We say $(M^n_0, g, d, o)$ is a compact Riemannian manifold with smooth boundary and a single conical singularity at $o \in M_0 \setminus \partial M_0$, if 
\begin{enumerate}[(\romannumeral1)]
\item $d$ is a metric on $M_0$ and $(M_0, d)$ is a compact metric space with smooth boundary,
\item $g$ is a smooth Riemannian metric on the regular part $ M_0 \setminus \{ o \}$, 
         $d$ is the induced metric by the Riemannian metric $g$ on $M_0 \setminus \{ o \}$ ,
\item there exists a neighborhood $U_o$ of $o$ in $M \setminus \partial M$, such that 
         $U_o \setminus \{ o \} \simeq (0, 1) \times N$ 
          for a smooth compact manifold $N$, 
           and on $U_o \setminus \{ o \}$ the metric $g = \g + h $, where
         \[
         \g = dr^2 + r^2 g^N,
         \]
         $g^N$ is a smooth Riemannian metric on $N$, $r$ is a coordinate on $(0, 1)$, $r=0$ corresponding the singular point $o$, 
         and $h$ satisfies 
         \begin{equation}\label{eqn-cone-metric-asymptotic}
         |\on^k h|_{\g} = O(r^{ \alpha - k }),  \ \ \text{as} \ \ r \rightarrow 0,
         \end{equation}
         for some $\alpha >0$ and $k = 0, 1 $ and $2$, where $\on$ is the Levi-Civita connection of $\g$.
\end{enumerate} 
}
\end{defn}

\begin{defn}\label{defn-AF-conic-mfld}
{\rm
We say $(M^n, g, o)$ is an asymptotically flat manifold with a single isolated conical singularity at $o$, if $M^n = M_0 \cup M_\infty$ satisfies
\begin{enumerate}[(\romannumeral1)]
\item $( M_0, g|_{M_0 \setminus\{o\}}, o)$ is a compact Riemannian manifold with smooth boundary and a single conical singularity at $o$ 
         defined as in Definition \ref{defn-conic-mfld}, 
\item  $M_\infty$ is diffeomorphic to  $ \R^n \setminus B_{R}(0) $ for some $R > 0$, and under this diffeomorphism the smooth Riemannian metric $g$ on $M_\infty$ satisfies
         \[
          g = g_{\R^n} + O(\rho^{-\tau}), \ \ |(\nabla^{g_{\R^n}})^i g|_{g_{\R^n}} = O(\rho^{-\tau - i}),  \ \  \text{as} \ \ \rho \rightarrow +\infty,
         \]
         for $i = 1, 2$ and $3$,
         where $\tau > \frac{n-2}{2}$ is the asymptotical order, $\nabla^{g_{\R^n}}$ is the Levi-Civita connection of the Euclidean metric $g_{\R^n}$, and $\rho$ is the Euclidean 
         distance to a base point.
\end{enumerate}
}
\end{defn}

\begin{rmrk}
{\rm
For the sake of simplicity of notations, in Definition \ref{defn-AF-conic-mfld} we only defined AF manifolds with a single conical singularity and one AF end, and we will only focus on this case in this paper. AF manifolds with finitely many isolated conical singularities and finitely many AF ends can be defined similarly, and all results in this paper can be easily extended to the case of finitely many isolated conical singularities and finitely many AF ends.
}
\end{rmrk}

Now we give the definition of the ADM mass of such Riemannian manifold.
\begin{defn}\label{def-mass}
{\rm
Let $(M^n, g, o)$ be an asymptotically flat manifold with a single isolated conical singularity at $o$. The mass $m(g)$ is defined as:
 \[m(g)=\lim_{R\to\infty}\frac{1}{\omega_{n}}\int_{S_{R}}(\partial_{i} g_{ji}-\partial_{j}g_{ii})\ast dx^{j},\]
 where $\{\frac{\partial}{\partial x^{i}}\}$ is an orthonormal basis of $g_{\mathbb{R}^{n}}$ and the $\ast$ operator is the Hodge star operator on the Euclidean space, the indices $i,j$ run over $M^n$ and $S_{R}$ is the sphere of radius $R$ on $\mathbb{R}^{n}$ and $\omega_{n}$ is the volume of the unit sphere in $\mathbb{R}^{n}$.
 }
\end{defn}

Finally we collect some basic facts for our model cone.
Let $(N^{n-1}, g^N)$ be a compact Riemannian manifold, and $(C(N), \g) = (\R_+ \times N, dr^2 + r^2 g^N)$ be the Riemannian cone over $(N^{n-1}, g^N)$, where $r$ is the coordinate on $\R_+$. We call $(N, g^N)$ the cross section of the cone. Let $\{ e_1, \cdots, e_{n-1} \}$ be a local orthonormal frame of $TN$ with respect to $g^N$, $\e_i = \frac{1}{r} e_i$ for all $i=1, \cdots, n-1$, and $\partial_r = \frac{\partial}{\partial r}$. Then $\{\e_1, \cdots, \e_{n-1}, \partial_r\}$ is a local orthonormal frame of $TC(N)$ with respect to $\g$. Let $\on$ denote the Levi-Civita connection of $\g$. By Koszul's formula, one easily obtains
\be\label{eqn-connection-cone}
\begin{cases}
 & \on_{\e_i} \pr = \frac{1}{r} \e_i, \\
 & \on_{\pr} \e_i = [ \pr, \e_i ] + \on_{\e_i} \pr = - \frac{1}{r} \e_i + \frac{1}{r} \e_i =0, \\
 & \on_{\pr} \pr = 0, \\
 & \on_{\e_i} \e_j = \nabla^{g^N}_{\e_i} \e_j  - \frac{1}{r} \delta_{ij} \pr,
\end{cases}
\ \ \forall 1 \leq i, j \leq n-1.
\ee
Let $\overline{R}$ denote the Riemann curvature tensor of $\g$, and $R^N$ denote the Riemann curvature tensor of $g^N$ on $N$. In the local frame $\{\e_1, \cdots, \e_{n-1}, \partial_r\}$, the only non-vanishing components of $\overline{R}$ are
\be
\overline{R}_{ijkl} = r^2 \left( R^N_{ijkl} + g^N_{ik} g^N_{jl} - g^N_{il} g^N_{jk} \right), \ \ 1 \leq i, j, k, l \leq n-1.
\ee
In particular,
\be\label{eqn-curvature-cone}
\overline{R} (\pr, X, Y, Z) =0, \ \ \forall X, Y, Z \in \Gamma(TC(N)),
\ee
and so
\begin{eqnarray*}
  \Ric_{\g}(X,Y) &=& \Ric_{g^{N}}(X,Y)-(n-2)g^{N}(X,Y),\ \ \forall X, Y \in \Gamma(TN);\\
  \Ric_{\g}(\pr,\cdot) &=& 0,
\end{eqnarray*}
and
\be\label{eqn-scalar-curvature-cone}
\Sc_{\g} =\frac{\Sc_{g^{N}}-(n-1)(n-2)}{r^2}.
\ee
The Laplace operator $\Delta_{\g}$ on $(C(N), \g)$ is given by
\begin{equation}\label{eqn-Delta-DeltaN}
\Delta_{\g} =\partial_{r}^2 + \frac{n-1}{r}\partial_{r}+\frac{1}{r^2}\Delta_{g^{N}},
\end{equation}
where $\Delta_{g^N}$ is Laplace operator on the cross section $(N, g^N)$.
\section{Analysis on AF manifolds with isolated conical singularities}\label{sect-analysis}

\subsection{Weighted Sobolev spaces}\label{subsect-weighted-Sobolev-space}
Let $(M^n, g, o)$ be a AF manifold with a single conical singularity at $o$ as defined in Definition \ref{defn-AF-conic-mfld}. Choose three cut-off functions $0 \leq \chi_1, \chi_2, \chi_3 \leq 1$ satisfying:
\be
\chi_1 (x) =
\begin{cases}
1, &  {\rm dist}(x, o) < \epsilon, \\
0, & {\rm dist}(x, o) > 2 \epsilon,
\end{cases}
\ee
\be
\chi_2(x) =
\begin{cases}
1, & {\rm dist}(x, o) > 2R, \\
0, & {\rm dist}(x, o) < R,
\end{cases}
\ee
and
\be
\chi_3 = 1 - \chi_1 - \chi_2,
\ee
where $\epsilon >0$ is chosen sufficiently small such that the ball $B_{2\epsilon}(o)$, centered at singular point $o$ with radius $2\epsilon$, is contained in the asymptotically conical neighborhood $U_o$ in Definition \ref{defn-conic-mfld}, and $R>0$ is chosen sufficiently large such that $M \setminus B_{R}(o) \subset M_\infty$.

For each $1 \leq p < +\infty, k \in \mathbb{N} $ and $\delta, \beta \in \R$,  the {\em weighted Sobolev space} $W^{k, p}_{\delta, \beta}(M)$ is defined to be the completion of $C^\infty_0(M\setminus \{o\})$ with respect to the {\em  weighted Sobolev norm } given by
\be
\| u \|^p_{W^{k, p}_{\delta, \beta}(M)} :=  \int_{M} \sum^{k}_{i=0} \left( r^{- p(\delta - i) -n} |\nabla^i u |^p \chi_1 + \rho^{ - p( \beta - i) -n} |\nabla^i u |^p \chi_2 + |\nabla^i u|^p \chi_3 \right) d\vol_{g} ,
\ee
where $r$ is the radial coordinate on the conical neighborhood $U_o$ in Definition \ref{defn-conic-mfld} and $\rho$ is the Euclidean distance function to a base point on $M_\infty$ in Definition \ref{defn-AF-conic-mfld}, and $|\nabla^i u|$ is the norm of $i^{th}$ covariant derivative of $u$ with respect to $g$.

Note that by the definition of the weighted Sobolev norms, we clearly have
\be
\begin{cases} \delta^\prime \ge \delta, \cr \beta^\prime \leq \beta, \end{cases}
\Rightarrow \ \  W^{k, p}_{\delta^\prime, \beta^\prime} \subset W^{k, p}_{\delta, \beta},
\ee
and 
\begin{equation}\label{eqn-polynomial-weighted-Sobolev-condition}
r^{\nu} \chi_1  \in W^{k, p}_{\delta, \beta}(M)  \Longleftrightarrow \nu > \delta, 
\quad 
\rho^{\mu} \chi_2 \in W^{k, p}_{\delta, \beta}(M)  \Longleftrightarrow \mu < \beta.
\end{equation}

\subsection{Elliptic estimate for the Laplace operator}

Throughout the paper, we always use $\lambda_j$ to denote the eigenvalues of the Laplace operator $\Delta_{g^N}$ on the cross section of the model cone $(C(N), \overline{g})$, and $E_j$ to denote the space of the corresponding eigenfunctions. We set
\begin{equation}\label{eqn-nu-defn}
\nu_j^{\pm} := \frac{-(n-2)\pm\sqrt{(n-2)^2+4\lambda_{j}}}{2}.
\end{equation}

\begin{defn}\label{defn-critical-cone}
{\rm
We say $\delta \in \R$ is {\em critical at conical point } if $\delta = \nu_j^{+}$ or $\delta = \nu_j^{-}$ for some $j \in \mathbb{Z}_{\geq 0}$, where $\nu^{\pm}_j$ is defined as in (\ref{eqn-nu-defn}).
}
\end{defn}

We view the Euclidean space $(\R^n, g_{\R^n})$ as a cone over the round sphere $\mathbb{S}^{n-1}$ with constant sectional curvature $1$, whose Laplace operator has eigenvalues: $j(n-2+j)$, $j \in \N$. Replacing $\lambda_j$ by these eigenvalues on round sphere, we make following definition.

\begin{defn}\label{defn-critical-infinity}
{\rm
We  say $\beta \in \R$ is {\em critical at infinity } if $\beta = k $ or $\beta = 2-n-k$ for some $ k \in \Z_{\geq 0} $.
}
\end{defn}

\begin{defn}\label{defn-critical-index}
{\rm
We say $(\delta, \beta) \in \R^2$ is {\em critical} if $\delta$ is critical at conical point or $\beta$ is critical at infinity.
}
\end{defn}

By using scaling technique, the usual interior elliptic estimates, and the asymptotic control of metric $g$ near conical point in Definition $\ref{defn-conic-mfld}$ and near infinity in Definition $\ref{defn-AF-conic-mfld}$, one can obtain the following weighted elliptic estimate. This is treated in detail in \cite{DSW-spin-23} for the Dirac operator and carries over except minor modifications.
\begin{prop}\label{prop-weight-elliptic-estimate}
Let $(M^n, g, o)$ be an AF manifold with a single conical singularity at $o$.
For any $1 \leq p < +\infty$, and $\delta, \beta \in \R$, if $u \in L^{p}_{\delta, \beta}(M)$, and $\Delta u \in L^{p}_{\delta-2, \beta-2}(M)$, then
\be
\|u \|_{W^{2, p}_{\delta, \beta}(M)} \leq C \left( \|\Delta u \|_{L^{p}_{\delta-2, \beta-2}(M)} + \|u\|_{L^{p}_{\delta, \beta}(M)} \right)
\ee
holds for some constant $C = C(g, n, p)$ independent of $u$.
\end{prop}

To obtain Fredholm property for the Laplace operator, the following refined weighted elliptic estimate is crucial.
\begin{prop}\label{prop-refined-weighted-elliptic-estimate}
Let $(M^n, g, o)$ be an AF manifold with a single conical singularity at $o$.
If $(\delta, \beta) \in \R^2$ is not critical as in Definition $\ref{defn-critical-index}$, there exists a constant $C= C(g, n)$ and a compact set $B \subset M \setminus \{o\}$ such that for any $u \in L^2_{\delta, \beta}(M)$ with $\Delta u \in L^{2}_{\delta-2, \beta-2}(M)$,
\be
\|u\|_{W^{2, 2}_{\delta, \beta}(M)} \leq C \left( \|\Delta u \|_{L^{ 2}_{\delta-2, \beta-2}(M)} + \|u\|_{L^{2}(B)} \right).
\ee
\end{prop}

Proposition $\ref{prop-refined-weighted-elliptic-estimate}$ is an immediate consequence of Proposition $\ref{prop-weight-elliptic-estimate}$ and the following Lemma:

\begin{lem}\label{lem-refined-weighted-elliptic-estimate}
Let $(M^n, g, o)$ be an AF manifold with a single conical singularity at $o$.
If $(\delta, \beta) \in \R^2$ is not critical as in Definition $\ref{defn-critical-index}$, there exists a constant $C$ and a compact set $B \subset M \setminus \{o\}$ such that for any $u \in L^2_{\delta, \beta}(M)$ with $\Delta u \in L^{2}_{\delta-2, \beta-2}(M)$,
\be
\|u\|_{L^{2}_{\delta, \beta}(M)} \leq C \left( \|\Delta u \|_{L^{2}_{\delta-2, \beta-2}(M)} + \|u\|_{L^{2}(B)} \right).
\ee
\end{lem}
To prove Lemma \ref{lem-refined-weighted-elliptic-estimate}, it is suffices to deal with model cone metric $\g$ on the conical neighborhood of the singular point $o$ and $g_{\R^n}$ near the AF infinity. The estimate near the AF infinity has been done in \cite{Minerbe-CMP}, by doing a spectral decomposition with respect to $\Delta_{g_{\mathbb{S}^{n-1}}}$, and reducing the problem to estimating coefficient functions, which satisfy certain second order ODEs. The Euclidean metric $g_{\R^n}$ can be viewed as a model cone metric with standard sphere as the cross section. The proof in \cite{Minerbe-CMP} can be adapted to that near the tip of a cone, and then the estimate near the conical singularity follows. In \cite{DSW-spin-23}, we have derived in detail the analogous estimate for the Dirac operator. So we omit the details of the proof, and refer to \cite{Minerbe-CMP} and \cite{DSW-spin-23}.

\subsection{Fredholm property of Laplace operator}\label{subsect-Fredholm}
We consider the unbounded operator
\begin{eqnarray*}
\Delta_{\delta, \beta}: {\rm Dom}(\Delta_{\delta, \beta}) & \rightarrow & L^2_{\delta-2, \beta-2}(M) \\
 \varphi & \mapsto & \Delta \varphi,
\end{eqnarray*}
whose domain ${\rm Dom}(\Delta_{\delta, \beta})$ is dense subset of $L^2_{\delta,\beta}(M)$ consisting of function $u$ such that $\Delta u \in L^2_{\delta-2, \beta-2}(M)$ in the sense of distributions.

As Lemma 4.8 in \cite{DSW-spin-23}, we have:
\begin{lem}
The unbounded operator $\Delta_{\delta, \beta}$ is closed.
\end{lem}

The usual $L^2$ pairing $(\cdot, \cdot)_{L^2(M)}$ identifies the topological dual space of $L^2_{\delta, \beta}$ with $L^{2}_{-\delta-n, - \beta -n}$. 
By this identification, the adjoint operator $\left( \Delta_{\delta, \beta} \right)^*$ of $\Delta_{\delta, \beta}$ is given as
\begin{eqnarray*}
\left( \Delta_{\delta, \beta} \right)^* : {\rm Dom}\left((\Delta_{\delta, \beta})^{*}\right) & \rightarrow & L^{2}_{-\delta-n, -\beta - n} \\
  u & \mapsto & \Delta u,
\end{eqnarray*}
where the domain: ${\rm Dom}\left( \left( \Delta_{\delta, \beta} \right)^* \right)$ is the dense subset of $L^{2}_{-\delta+2-n, -\beta+2-n}$ consisting of functions $u$ such that $\Delta u \in L^2_{-\delta-n, -\beta-n}$ in the distributional sense.

By applying the refined weighted elliptic estimate in Proposition \ref{prop-refined-weighted-elliptic-estimate}, as Proposition 4.9 in \cite{DSW-spin-23}, we have the following Fredholm property.
\begin{prop}\label{prop-Fredholm}
If $(\delta, \beta) \in \R^2$ is not critical as in Definition $\ref{defn-critical-index}$, then the operator $\Delta_{\delta, \beta}$ is Fredholm, namely, 
\begin{enumerate}[$(1)$]
\item ${\rm Ran}(\Delta_{\delta, \beta})$ is closed,
\item ${\rm dim}\left( {\rm Ker}(\Delta_{\delta, \beta}) \right) < +\infty$,
\item ${\rm dim}\left( {\rm Ker}((\Delta_{\delta, \beta})^*) \right) < +\infty$.
\end{enumerate}
\end{prop}
For the proof one can see that of Proposition 4.9 in \cite{DSW-spin-23}.

\subsection{Solving the Laplace equation}

Let $(M^n, g, o)$ be an AF manifold with a single conical singularity at $o$ as defined in \ref{defn-AF-conic-mfld}. The equation $\Delta u=f$ has been solved near infinity in \cite{Minerbe-CMP}. Here we similarly solve it near the conical point $o$. The strategy is similar as solving the Dirac equation near the cone point in \cite{DSW-spin-23}, and so we omit proofs of some results in this section and refer to the analogous results for Dirac operator in Section 4 in \cite{DSW-spin-23}.

We use $B_r$ to denote the ball with radius $r$ centered at the conical singularity $o$. In the notion of weighted Sobolev norms and spaces over $B_r$, the subscript $\beta$ will be neglected, and they will be written as $\|\cdot\|_{W^{2, 2}_{\delta}(B_r)}$ and $W^{2, 2}_{\delta}(B_r)$, since there is no asymptotic control near infinity need to be concerned over the finite ball $B_r$.

For the model cone metric $\g = dr^2 + r^2 g^N$, by solving the Dirichlet problem for the equation $\Delta_{\g} u = f$ on a compact exhaustion of $B_{2r_0} \setminus \{o\}$, and using refined weighted elliptic estimate in Lemma \ref{lem-refined-weighted-elliptic-estimate}, one can solve $\Delta_{\g}u = f$ on $B_{2r_0}$ and prove the following:
\begin{lem}\label{lemma-G}
  For $\delta \in \R$, which is noncritical at conical point, and a small number $r_{0} >0$, there is a bounded operator
  \[G_{\g}:L^{2}_{\delta-2}(B_{2r_{0}})\to W^{2,2}_{\delta}(B_{2r_{0}})\]
   such that $\Delta_{\g}\circ G_{\g}=id$.
\end{lem}


A perturbation argument extends this result to a more general setting, namely, for an asymptotically conical metric $g = \g + h$ where $h$ satisfies (\ref{eqn-cone-metric-asymptotic}), and we have the following:
\begin{prop}\label{prop-G-g}
  For $\delta \in \R$, which is noncritical at cone point as in \defref{defn-critical-cone}, and a small number $r_{0}$, there is a bounded operator $G_g:L^{2}_{\delta-2}(B_{2r_{0}})\mapsto W^{2,2}_{\delta}(B_{2r_{0}})$ such that $\Delta_g \circ G_g =id$.
\end{prop}


As an application of Proposition $\ref{prop-G-g}$, one can solve harmonic functions near the conical point, which are asymptotic to harmonic functions, $r^{\nu^\pm_j}\phi_j$, with respect to the model cone metric $\g$.

\begin{cor}\label{cor-harmonic-spinor-on-cone}
Given $j\in \mathbb{N}$ and $\phi\in E_{j}$, there are functions $\mathcal{H}^{\pm}_{j,\phi}$ that are harmonic near the conical point and can be written as \[\mathcal{H}^{\pm}_{j,\phi}=r^{\nu^{\pm}_{j}}\phi+v_{\pm}\]
with $v_{+}$ in $W^{2,2}_{\eta}$ for any $\eta<\nu^{+}_{j}+\alpha$ and $v_{-}$ in $W^{2,2}_{\eta}$ for any $\eta<\nu^{-}_{j}+\alpha$.
\end{cor}


The solution of the equation $\Delta u = f$ may not be unique, and the difference of two solutions is a harmonic function. For the model cone metric $\overline{g} = dr^2 + r^2 g^N$, the Laplace operator $\Delta_{\overline{g}}$ is given in (\ref{eqn-Delta-DeltaN}), and the harmonic functions are linear combinations of $r^{\nu^{\pm}_j} \phi_j$, where $\nu^{\pm}_j$ are given in (\ref{eqn-nu-defn}) and $\phi_j \in E_j$, i.e. $\Delta_{g^N} \phi_j = \lambda_j \phi_j$. Based on this observation, by using (\ref{eqn-polynomial-weighted-Sobolev-condition}) and \lemref{lemma-G}, one can prove the following:

\begin{lem}\label{lem-decay-jump-for-Dbar}
  Suppose $\Delta_{\g} u=f$ with $u$ in $L_{\delta}^2(B_{2r_{0}})$ and $f$ in $L^{2}_{\delta'-2}(B_{2r_{0}})$ for non-critical exponents $\delta<\delta'$ and a small number $r_{0}$. Then there is an element $v$ of $L^{2}_{\delta'}(B_{2r_{0}})$ such that $u-v$ is a linear combination of functions: 
     $r^{\nu_{j}^{\pm}}\phi_{j}$ with $\phi_{j} \in E_{j}$ and $\delta<\nu^{\pm}_{j}<\delta'$;
\end{lem}

By using Corollary \ref{cor-harmonic-spinor-on-cone}, the property for model cone metrics in Lemma \ref{lem-decay-jump-for-Dbar} can be extended to asymptotically conical metrics $g$ as following:
\begin{prop}\label{prop-asmptotic-order} 
  Suppose $\Delta_g u=f$ with $u$ in $L_{\delta}^2(B_{2r_{0}})$ and $f$ in $L^{2}_{\delta'-2}(B_{2r_{0}})$ for non-critical exponents $\delta<\delta'$ . Then, up to making $B_{2r_{0}}$ smaller, there is an element $v$ of $L^{2}_{\delta'}(B_{2r_{0}})$ such that $u-v$ is a linear combination of the following functions:
    $\mathcal{H}^{\pm}_{j,\phi_{j}}$ with $\phi_{j}$ in $E_{j}$ and $\delta<\nu^{\pm}_{j}<\delta'$,
  where 
  $ \mathcal{H}^{\pm}_{j,\phi_{j}} $ are functions, harmonic near the conical singularity, obtained in Corollary \ref{cor-harmonic-spinor-on-cone}.
\end{prop}


Now we are ready to prove the surjectivity of $\Delta_{\delta, \beta}$ for certain noncritical indices $(\delta, \beta)$, which enables us to solve the equation $\Delta_g u = f$.

First, by using Fredholm property in \propref{prop-Fredholm}, we can prove the following surjectivity. For the proof, we refer to proofs of Corollary 2 in \cite{Minerbe-CMP} and Proposition 4.15 in \cite{DSW-spin-23}.

\begin{prop}\label{prop-surjectivity-1}
Let $(M^n, g, o)$ be an AF manifold with a single conical singularity at $o$.
For any noncritical $(\delta, \beta)$ satisfying $\delta \leq \frac{2-n}{2}$ and $\beta \geq \frac{2-n}{2}$,
the map
\be
\Delta_{\delta, \beta}: {\rm Dom}\left( \Delta_{\delta, \beta} \right) \rightarrow L^2_{\delta -2, \beta -2}(M)
\ee
is surjective.

Moreover, the map
\be
\Delta_{\frac{2-n}{2}, \frac{2-n}{2}}: {\rm Dom}\left( \Delta_{\frac{2-n}{2}, \frac{2-n}{2}} \right) \rightarrow L^{2}_{-\frac{n+2}{2}, -\frac{n+2}{2}}(M)
\ee
is an isomorphism.
\end{prop}


Then by applying Propositions \ref{prop-asmptotic-order} and \ref{prop-surjectivity-1}, we can extend the region of indices $(\delta, \beta)$ for which $\Delta_{\delta, \beta}$ is surjective, and obtain the following: 
\begin{prop}\label{prop-surjectivity}
Let $(M^n, g, o)$ be an AF manifold with a single conical singularity at $o$. We have that
\be
\Delta_{\delta, \beta}: {\rm Dom}\left( \Delta_{\delta, \beta} \right) \rightarrow L^{2}_{\delta-2,\beta-2}(M)
\ee
is surjective for noncritical $\delta< 0$ and $\beta>2-n$.
\end{prop}
\begin{proof}
By Proposition \ref{prop-surjectivity-1}, it suffices to show that $\Delta_{\delta, \beta}$ is surjective for $\frac{2-n}{2} < \delta < 0$ and $2-n < \beta < \frac{2-n}{2}$.

For arbitrary noncritical $\delta^\prime \in \left( \frac{2-n}{2}, 0 \right)$ and $\beta^\prime \in \left( 2-n, \frac{2-n}{2}\right)$, we take an arbitrary function $ f \in L^2_{\delta^\prime -2, \beta^\prime -2}(M) \subset L^2_{\frac{2-n}{2}-2, \frac{2-n}{2}-2}(M)$. Proposition \ref{prop-surjectivity-1} then implies that there exists $u \in L^2_{\frac{2-n}{2}, \frac{2-n}{2}}(M)$ such that $\Delta u = f$. Proposition \ref{prop-asmptotic-order} then implies that $ u \in L^2_{\delta^\prime, \beta^\prime}$, since there is no critical index at  conical point in $\left( \delta^\prime,0 \right)$ and no critical index at infinity in $\left( \beta^\prime, \frac{2-n}{2} \right)$. Therefore, $\Delta_{\delta^\prime, \beta^\prime}$ is surjective, and this completes the proof.
\end{proof}


\section{Nonnegativity of mass}\label{sec-nonnegativity}

In this section, we prove the non-negativity of the ADM mass for AF manifolds with a single conical singularity with nonnegative scalar curvature on the regular part, by following the approach of Ju and Viaclovsky in the case of orbifold singularity in \cite{TV2023}. For that, we first solve a harmonic function as following:
\begin{lem}\label{lem-harmonic-function}
Let $(M^n, g, o)$ be a $n$-dimensional AF manifold with a single conical singularity at $ o$. There exists a harmonic function $u$ on $ M $ which satisfies $u > 1$ and admits the expansion
\begin{equation}\label{eqn-harmonic-function-asymptotic}
u = \begin{cases}
      r^{2-n} + o(r^{2-n+\alpha^\prime}), & \text{as} \ \ r \to 0, \\
      1 + A \rho^{2-n} + o(\rho^{2-n - \epsilon}), & \text{as} \ \ \rho \to \infty,
      \end{cases}
\end{equation}
for $ 0< \alpha^\prime < \min\{1, \alpha\}$ and $ 0 < \epsilon < \min\{1, \tau\}$, and some constant $ A>0$. Here $\alpha$ and $ \tau $ are asymptotic orders of the metric $g$ in Definitions \ref{defn-conic-mfld} and \ref{defn-AF-conic-mfld}.
\end{lem}

\begin{proof}
Choose a cut-off function $\phi$ satisfying
\begin{equation*}
\phi = \begin{cases}
        1, & \text{on} \ \ B_{\frac{1}{2}}(o), \\
        0, & \text{on} \ \ M \setminus B_1(o).
       \end{cases}
\end{equation*}
Let $u_0 = \phi r^{2-n}$, which is a smooth function supported in a neighborhood of the conically singular point $o$. Then because $\Delta_{\g} r^{2-n} = 0$, by the asymptotic control of $g$ near conically singular point $o$, we have 
\begin{equation*}
\Delta_g u_0 = \begin{cases}
             O(r^{-n+ \alpha}), & \text{as} \ \ r \to 0, \\
             0, & \text{on} \ \ M \setminus B_1(o).
             \end{cases}
\end{equation*}
By (\ref{eqn-polynomial-weighted-Sobolev-condition}), this implies that 
\begin{equation*}
\Delta_g u_0 \in L^2_{\delta-2, \beta-2}(M), \quad \forall \delta < 2-n+\alpha, \ \ \beta > 2-n.
\end{equation*}
Then by applying Proposition \ref{prop-surjectivity}, we obtain $v \in L^{2}_{\delta, \beta}$ such that
\begin{equation*}
\Delta_g v = \Delta_g u_0.
\end{equation*}
So by setting $u_1 = u_0 - v$, we have $\Delta_g u_1 = 0$. 

Now we derive the asymptotic behavior of $v$ near conically singular point and AF infinity. Near the singular point $o$, $u_1 = r^{2-n} - v$, and so $\Delta_g u_1 = 0$ implies
\begin{equation*}
\Delta_g v = \Delta_g r^{2-n} = \left( \Delta_g - \Delta_{\g} \right) r^{2-n} = O(r^{-n+\alpha}) \in L^{p}_{\delta -2}(B_{\frac{1}{2}}(o)), \ \ \forall \delta < 2-n+\alpha \ \ \text{and} \ \  \forall p >1,
\end{equation*}
by \eqref{eqn-polynomial-weighted-Sobolev-condition}.
Then the weighted elliptic estimate in \propref{prop-weight-elliptic-estimate} implies that $v \in W^{2, 2}_{\delta}(B_{\frac{1}{2}}(o))$ for  $\delta < 2-n+\alpha$. Consequently, the weighted Sobolev inequality in Proposition 3.4 in \cite{DW-MRL-2020} implies that $v \in L^{\frac{2n}{n-2}}_{\delta}(B_{\frac{1}{2}}(o))$, and by applying weighed elliptic estimate in \propref{prop-weight-elliptic-estimate} (with $p = \frac{2n}{n-2}$) again, we obtain $v \in W_\delta^{2, \frac{2n}{n-2}}(B_{\frac{1}{2}}(o))$ for $ \delta < 2-n+\alpha$. By repeating this process, we can obtain that $v \in W^{2, p}_{\delta}(B_{\frac{1}{2}}(o))$ for all $p>1$ and $\delta<2-n+\alpha$, and so, by Proposition 3.4 in \cite{DW-MRL-2020}, we have
\begin{equation}
v = o(r^{2-n+\alpha^\prime}), \ \ \text{as} \ \ r \to 0,
\end{equation}
for $0< \alpha^\prime < \alpha$.

Near the AF infinity, $u_1 = v \in L^{2}_{\beta}(M_\infty)$ for $\beta > 2-n$, and $\Delta_g v = \Delta_g u_1 =0 \in L^{2}_{\beta^\prime}(M_\infty)$ for $1-n < \beta^\prime < 2-n$. Thus, Proposition 4 in \cite{Minerbe-CMP} implies that 
\begin{equation*}
v = A \rho^{2-n} + v^\prime, \ \ v^\prime \in L^{2}_{\beta^\prime}(M_\infty),
\end{equation*}
for $1-n < \beta^\prime < 2-n$, and some constant $A$. Moreover,
\begin{equation*}
\Delta_g v^\prime = - \Delta_g (A \rho^{2-n}) = (\Delta_{g_{\R^{n}}} - \Delta_g) (A \rho^{2-n}) = O(\rho^{-n-\tau}) \in L^{p}_{\beta^\prime - 2}(M_\infty),
\end{equation*}
for $(2-n-\tau <) 1- n < \beta^\prime < 2-n$, and all $p>1$. Then the weighted elliptic estimate in Proposition \ref{prop-weight-elliptic-estimate} implies $v^\prime \in W^{2, 2}_{\beta^\prime}(M_\infty)$. Then by a similar weighted elliptic bootstrapping argument as above, we can obtain $v^\prime \in W^{2, p}_{\beta^\prime}(M_\infty)$ for all $p>1$ and $1-n < \beta^\prime < 2-n$, and so
\begin{equation}
v^\prime = o(\rho^{2-n-\epsilon}), \ \ \text{as} \ \ \rho \to \infty,
\end{equation}
for $0 < \epsilon < 1$.

In summary, we obtain a harmonic function $u_1$ admitting the asymptotic behavior:
\begin{equation}
u_1 = \begin{cases}
       r^{2-n} + o(r^{2-n+\alpha^\prime}), & \text{as} \ \ r \to 0, \\
       A\rho^{2-n} + o(\rho^{2-n-\epsilon}), & \text{as} \ \ \rho \to \infty,
      \end{cases}
\end{equation}
for $0 < \alpha^\prime < \alpha$ and $0 < \epsilon < \min\{1, \tau\}$.

We let $u := 1 + u_1$. Then $u$ is a harmonic function satisfying the asymptotic control in (\ref{eqn-harmonic-function-asymptotic}). Finally, by using the strong maximum principle and the asymptotic behavior in (\ref{eqn-harmonic-function-asymptotic}), we obtain $u \geq 1$, and in particular, $A > 0$. Then by applying the strong maximum principle again, we can obtain $u > 1$.
\end{proof}

Now we are ready to prove the following:
\begin{thm}\label{thm-nonnegativity-of-mass}
Let $(M^n, g, o), n \geq 3,$ be a $n$-dimensional AF manifold with a single conical singularity at $o$. If the scalar curvature $\Sc_g \geq 0$, then the ADM mass $m(g) \geq 0$.
\end{thm}
\begin{proof}
For the harmonic function obtained in Lemma \ref{lem-harmonic-function}, and any $\delta > 0$, we define
\begin{equation}
u_\delta = \delta u + (1 - \delta),
\end{equation}
which satisfies $\Delta_g u_\delta = 0$, $u_\delta >1$, and admits the asymptotic expansion
\begin{equation}\label{eqn-u-delta-asymptotic}
u_\delta = \begin{cases}
           \delta r^{2-n} + o(r^{2 - n + \alpha^\prime}), & \text{as} \ \ r \to 0, \\
           1 + \delta A r^{2-n} + o(\rho^{2-n-\epsilon}), & \text{as} \ \ \rho \to \infty,
           \end{cases}
\end{equation}
for $0 < \alpha^\prime < \alpha$ and $ 0 < \epsilon < 1$. Here the constant $A >0$.

Then we do a conformal change for the metric $g$, and let 
\begin{equation}
g_\delta := \left( u_{\delta} \right)^{\frac{4}{n-2}} g \ \ \text{on } \ \ M.
\end{equation}
By the asymptotic expansion of $u_\delta$ at infinity as in (\ref{eqn-u-delta-asymptotic}), $g_\delta$ is asymptotically flat of order $\min\{\tau, n-2\}$. Near the conically singular point $o$, the metric $g$ is given as
\begin{equation}
g = dr^2 + r^2 g^N + h,
\end{equation}
where $h$ satisfies (\ref{eqn-cone-metric-asymptotic}), and so 
\begin{equation}
g_\delta 
= \left( u_\delta \right)^{\frac{4}{n-2}} g 
= \delta^{\frac{4}{n-2}} r^{-4}\left( 1+ o(r^{\alpha^\prime})\right) \left(dr^2 + r^2 g^N + h \right), \ \ \text{as} \ \ r \to 0.
\end{equation}
Now we do a change of variable, by letting $s = \delta^{\frac{2}{n-2}}\frac{1}{r}$, and rewrite $g_\delta$ as
\begin{equation}
g_\delta = (1 + o(s^{-\alpha^\prime}))(ds^2 + s^2 g^N + \tilde{h}),
\end{equation}
where $|\tilde{h}|_{ds^2 + s^2 g^N} = O(s^{-\alpha})$ as $s \to \infty$. Therefore, 
\begin{equation}
g_\delta = ds^2 + s^2 g^N + o(s^{-\alpha^\prime}), \ \ \text{as} \ \ s \to \infty,
\end{equation}
since $0 < \alpha^\prime < \alpha$. Note that the metric $g_\delta$ is asymptotically conical, and it tends to the infinity end of a cone, as $s \to \infty$ (i.e. $r \to 0$). Moreover, the scalar curvature of $g_\delta$ is given by 
\begin{equation}
\Sc_{g_\delta} 
= \frac{4(n-1)}{n-2} \left( u_\delta \right)^{-\frac{n+2}{n-2}} \left( -\Delta_g u_\delta + \frac{n-2}{4(n-1)}\Sc_g u_\delta \right) 
= \left( u_\delta \right)^{-\frac{4}{n-2}} \Sc_g \geq 0.
\end{equation}
For any fixed $\delta >0$, we can choose a hypersurface $\Sigma_{\delta, s_0} : = \{s = s_0\}$ for some sufficiently large $s_0$ such that $\Sigma_{\delta, s_0}$ is strictly mean concave with respect to the normal pointing to the AF end of $(M, g_\delta)$. Let $M_\delta = M \setminus \{s \geq s_0\}$ be a AF manifold with strictly mean concave boundary. By applying Theorem in \cite{Hirsch-Miao-20}, we obtain
\begin{equation}
m(M_\delta, g_\delta) \geq 0.
\end{equation}
On the other hand,
\begin{equation}
m(M_\delta, g_\delta) = m(M, g) + 4(n-2)\delta A.
\end{equation}
Because this is true for any constant $\delta >0$, and $A$ is a fixed constant, we must have 
\begin{equation}
m(M, g) \geq 0.
\end{equation}
\end{proof}

\begin{rmrk}
{\rm
Note that $(M, g_\delta)$ is a complete Riemannian manifolds with two ends, one is an AF end and another is an asymptotically conical end. In the case of the dimension $3 \leq n \leq 7$, the results of the positive mass theorems on manifold with arbitrary ends in \cite{LUY-JDG, Zhu-IMRN-23} imply the mass $m(M, g_\delta) \geq 0$. Then our conclusion follows as well.
}
\end{rmrk}

\section{Rigidity}\label{sec-rigidity}
In this section, we prove that when the mass is zero, then $M$ is isometric to $\mathbb{R}^{n}$. We first prove that $M$ is Ricci flat, then show that it must be the Euclidean space. 

First of all, we state the following Sobolev inequality on AF manifold with a single conical singularity, which can be obtained by combining the Sobolev inequality on cone (see, Lemma 3.2 in \cite{DW-MRL-2020}) and that on the Euclidean space. It will be used in the proof of Lemma \ref{Lem-conformal-equation}.

\begin{lem}\label{lem-Sobolev-inequality}
Let $ (M^n, g, o) $ be a $n$-dimensional AF conical singularity at $o$. There is a Sobolev constant $C>0$ such that
\begin{equation}
\left( \int_{M} |f|^{\frac{2n}{n-2}} d\vol_{g} \right)^{\frac{n-2}{2n}} \leq C \left( \int_{M} |\nabla f|^2 d\vol_g \right)^{\frac{1}{2}}
\end{equation}
holds for all $f \in C^{\infty}_{0}(\mathring{M})$.
\end{lem}

Then as analysis preparations for proving the rigidity result in Theorem \ref{thm-main} (stated in Theorem \ref{thm-rigidity} below), we solve Schr\"odinger equation with compactly supported potential function on AF manifolds with a conical singularity in Lemma \ref{Lem-conformal-equation}; and solve harmonic functions, which are asymptotic to Euclidean coordinate near the AF infinity in Lemma \ref{lem-harmonic-coordinate}. We also study their asymptotic behaviors near the conical singularity, which are critical in the proof of Theorem \ref{thm-rigidity}.

\begin{lem}\label{Lem-conformal-equation}
Let $(M^n, g, o)$ be an AF manifold with a conical singularity at $o$. There exists a constant $\epsilon_0 = \epsilon(g) $, such that if $f$ is a smooth function with compact support in $\mathring{M}$ and $ \| f_{-} \|_{L^{\frac{n}{2}}(M)} < \epsilon_0$, then the equation
\begin{equation}\label{eqn-u}
\begin{cases}
    -\Delta_g u + f u =0,\\
    u\to 1 \mbox{ as } x\to\infty,
\end{cases}
\end{equation}
has a positive solution admitting the asymptotic as:
\begin{equation}\label{eqn-asymptotic-u-rigidity}
u = 
\begin{cases}
 1 + \frac{A}{\rho^{n-2}} + u_1 \ \ 
 & \text{as} \ \ \rho \to \infty, \\
 B + a(y) r^{\nu_1} + u_2 \ \ 
 & \text{as} \ \ r \to 0,
\end{cases}
\end{equation}
where $A$ and $B$ are constants, and $B \geq 1$, 
\begin{equation}
    \nu_1 := \frac{2-n + \sqrt{(n-2)^2 + 4 \lambda_1}}{2},
\end{equation}
here $\lambda_1$ is the first non-zero eigenvalue of the Laplacian $\Delta_{g^N}$ and $a(y)$ is a corresponding eigenfunction, i.e. $\Delta_{g^N} a(y) = - \lambda_1 a(y)$, and $u_1$ and $u_2$ have asymptotic as
\begin{equation}\label{eqn-asymptotic-u1}
    |\nabla^i u_1| = o(\rho^{\beta - i}), \ \ \text{as} \ \ \rho \to +\infty, \ \ \text{for} \ \ 2-n-\tau < \beta < 2-n, \ \ i = 0, 1,
\end{equation}
and
\begin{equation}\label{eqn-asymptotic-u2}
    |\nabla^i u_2| = o(r^{\delta - i}), \ \ \text{as} \ \ r \to 0, \ \ \text{for} \ \ \nu_1 < \delta < \nu_1 + \alpha, \ \ i=0, 1.
\end{equation}

\end{lem}
\begin{proof}
{\bf Existence of a positive solution}:
Let $v=1-u$. Then equation \eqref{eqn-u} becomes
 \begin{equation}\label{FNEQ}
      \Delta_{g} v-fv=-f.
   \end{equation}
   On a compact subset $A_{r,\rho}:= B_{\rho}(o) \setminus B_{r}(o)$, consider the Dirichlet problem
   \begin{equation}\label{NPEQ}
    \begin{cases}
      \Delta_{g} v_{r,\rho}-fv_{r,\rho}=-f, & \text{in}\; A_{r,\rho}, \\
      v_{r,\rho}=0,  & \text{on}\; \partial A_{r,\rho}.
    \end{cases}
   \end{equation}
   By Fredholm alternative, if the homogeneous equation
   \begin{equation}\label{UE}
    \begin{cases}
      \Delta_{g} v_{r,\rho}-fv_{r,\rho}=0, & \text{in}\; A_{r,\rho} \\
      v_{r,\rho}=0, & \text{on}\; \partial A_{r,\rho}
    \end{cases}
   \end{equation}
   has only zero solution, 
   then equation (\ref{NPEQ}) has a unique solution. Suppose $\omega$ is a solution of equation (\ref{UE}). Multiplying $\omega$ to both sides of the equation in (\ref{UE}) and integrating by parts, by the H\"older inequality with $p=\frac{n}{2},q=\frac{n}{n-2}$ and the Sobolev inequality in \lemref{lem-Sobolev-inequality}, we have
   \begin{eqnarray*}
     \int_{A_{r,\rho}}|\nabla \omega|^2 d\operatorname{vol}_{g}&=& -\int_{A_{r,\rho}}f\omega^2d\operatorname{vol}_{g}\leq \int_{A_{r,\rho}}f_{-}\omega^2d\operatorname{vol}_{g}\\
     &\leq&\left(\int_{A_{r,\rho}}f_{-}^{\frac{n}{2}}d\operatorname{vol}_{g}\right)^{\frac{2}{n}}\left(\int_{A_{r,\rho}}\omega^{\frac{2n}{n-2}}d\operatorname{vol}_{g}\right)^{\frac{n-2}{n}}  \\
      &\leq&c_{1} \left(\int_{A_{r,\rho}}f_{-}^{\frac{n}{2}}d\operatorname{vol}_{g}\right)^{\frac{2}{n}}\left(\int_{A_{r,\rho}}|\nabla \omega|^2d\operatorname{vol}_{g}\right)
   \end{eqnarray*}
   Thus if $\|f_{-}\|_{L^{\frac{n}{2}}(M)}< \frac{1}{c_{1}}$, then $\omega=0$. Therefore, equation (\ref{NPEQ}) has a unique solution $v_{r, \rho}$.
   Multiplying $v_{r,\rho}$ to both sides of (\ref{NPEQ}), using H\"older inequality and Sobolev inequality again,
   \begin{eqnarray*}
     \int_{A_{r,\rho}}|\nabla v_{r,\rho}|^2d\operatorname{vol}_{g} &\leq& \int_{A_{r,\rho}}f_{-}v^2_{r,\rho}d\operatorname{vol}_{g}+\int_{A_{r,\rho}}fv_{r,\rho}d\operatorname{vol}_{g} \\
      &\leq& c_{1}\left(\int_{A_{r,\rho}}f_{-}^{\frac{n}{2}}d\operatorname{vol}_{g}\right)^{\frac{2}{n}}\left(\int_{A_{r,\rho}}|\nabla v_{r,\rho}|^2d\operatorname{vol}_{g}\right)\\
      &&+\left(\int_{A_{r,\rho}}|f|^{\frac{2n}{n+2}}d\operatorname{vol}_{g}\right)^{\frac{n+2}{2n}}\left(\int_{A_{r,\rho}}v_{r,\rho}^{\frac{2n}{n-2}}d\operatorname{vol}_{g}\right)^{\frac{n-2}{2n}}\\
      &\leq&c_{1}\left(\int_{A_{r,\rho}}f_{-}^{\frac{n}{2}}d\operatorname{vol}_{g}\right)^{\frac{2}{n}}\left(\int_{A_{r,\rho}}|\nabla v_{r,\rho}|^2d\operatorname{vol}_{g}\right)\\
      &&+c_{1}\left(\int_{A_{r,\rho}}|f|^{\frac{2n}{n+2}}d\operatorname{vol}_{g}\right)^{\frac{n+2}{2n}}\left(\int_{A_{r,\rho}}| \nabla v_{r,\rho}|^{2}d\operatorname{vol}_{g}\right)^{\frac{1}{2}}
   \end{eqnarray*}
   As a consequence, there is a constant $c_{2}$ depending on $(M,g,f)$ such that $\|v_{r,\rho}\|_{L^{\frac{2n}{n-2}}}<c_{2}$ and $\|\nabla v_{r,\rho}\|_{L^2}<c_{2}$. The standard theory of elliptic equations concludes that $v_{r,\rho}$ have uniformly bounded $C^{2,\alpha}$ norms. By Arzela-Ascoli theorem we may pass to a limit and conclude that equation \eqref{FNEQ} has a solution $v$. Then $u := 1 + v$ is a solution of \eqref{eqn-u}.

   We claim that $u$ is a positive solution. First, we show that $u$ is nonnegative. Otherwise, $\Omega_{-}:= \{ x \in M \mid u(x) < 0\}$ is nonempty open set, and its boundary $\partial \Omega_{-} = u^{-1}(0)$ is a $(n-1)$-dimensional manifold, except on a closed set of lower dimension. We consider the Dirichlet boundary problem:
   \begin{equation}\label{eqn-u-nonnegative}
    \begin{cases}
      \Delta_{g} u-fu=0, & \text{in}\; \Omega_{-}, \\
      u= 0, &\text{on}\; \partial \Omega_{-}.
    \end{cases}
   \end{equation}
   Similarly as showing that \eqref{UE} has only zero solution, we can show that \eqref{eqn-u-nonnegative} has also only zero solution, i.e. $u = 0$ in $\Omega_{-}$. This contradicts with the construction of $\Omega_{-}$. Thus, $u$ must be nonnegative.
   Then by applying the strong maximum principle, one can see that $u$ must be positive everywhere.

   {\bf Asymptotics of the solution}: Because $f$ has compact support away from the singular point $o$, there exits sufficiently small $r_0 > 0$ such that $\Delta_g v_{r, \rho} = 0$ on $A_{r, r_0} = B_{r_0}(o) \setminus B_r(o)$ for all $0< r < r_0$. For a small $\epsilon > 0$, there exists a sufficient large constant $C_0 >0$ such that 
   $|v_{r, \rho}| \leq \frac{C_0}{r^{n-2-\epsilon}}$ on $\partial B_{r_0}(o)$ for all $r>0$, since $v_{r, \rho}$ have uniformly bounded $C^{2, \alpha}$-norm. Note that $C_0$ only depends on $r_0$ and $\epsilon$. Moreover,
   \begin{eqnarray}
   \Delta_g \left( C_0 r^{2-n+\epsilon} \right) 
   & = & \Delta_{\g}\left( C_0 r^{2-n+\epsilon} \right) + (\Delta_g - \Delta_{\g})\left( r^{2-n+\epsilon} \right) \\
   & = & \epsilon (2-n+\epsilon) C_0 r^{-n+\epsilon} + O(r^{-n+\epsilon + \alpha}),
   \end{eqnarray}
   and so we can choose $r_0$ sufficiently small so that $\Delta_g \left( C r^{2-n+\epsilon} \right) \leq 0$ on $B_{r_0}(o)$. Thus, we have
   \begin{equation}
   \begin{cases}
   \Delta_g \left( C_0 r^{2-n+\epsilon} \pm v_{r, \rho} \right) \leq 0, & \text{in} \ \ A_{r, r_0}, \\
   C_0 r^{2-n+\epsilon} \pm v_{r, \rho} \geq 0, & \text{on} \ \ \partial A_{r, r_0}.
   \end{cases}
   \end{equation}
   Then maximum principle implies 
   \begin{equation}
   C_0 r^{2-n+\epsilon} \pm v_{r, \rho} \geq 0, \ \ \text{on} \ \ A_{r, r_0}, \ \ \forall 0 < r < r_0.
   \end{equation}
   In other words, $|v_{r, r_0}| \leq C_0 r^{2-n+\epsilon}$ on $A_{r, r_0}$ for all $0 < r < r_0$. By taking limit, we conclude that a solution $v$ of \eqref{FNEQ} satisfies
   \begin{equation}
       |v| \leq C_{0}r^{2-n+\epsilon} \ \ \text{on} \ \ B_{r_0}(o).
   \end{equation}

Similarly, we can show that for sufficiently large $\rho_0>0$ and $C_1 > 0$, the solution $v$ of \eqref{FNEQ} satisfies
\begin{equation}
    |v| \leq C_1 \rho^{2-n+\epsilon} \ \ \text{on} \ \ M \setminus B_{\rho_0}(o).
\end{equation}

Consequently, we obtain that $v \in L^{2}_{\delta, \beta}$ for $2-n < \delta < 2-n + \epsilon <0$ and $ 2-n + \epsilon < \beta < 0$, by \eqref{eqn-polynomial-weighted-Sobolev-condition}. Note that
\begin{equation}
\Delta_g v = fv - f \in L^{2}_{\delta^\prime - 2, \beta^\prime - 2}, \ \ \text{for} \ \ \delta^\prime > 0 \ \ \text{and} \ \ \beta^\prime < 2-n,
\end{equation}
since $f$ vanishes near the conically singular point and near the AF infinity. Then because there is no critical index between $2-n$ and $0$ at either cone point or infinity, by applying Proposition \ref{prop-asmptotic-order} and the maximal principle, there are constants $A_1$ and $B_1$
such that 
\begin{equation}
    v = \begin{cases}
         \frac{A_1}{\rho^{n-2}} + u_1, & \text{as} \ \ \rho \to \infty, \\
        B_1 + a(y) r^{\nu_1} + u_2, & \text{as} \ \ r \to 0,
        \end{cases}
\end{equation}
where
\begin{equation}
    \nu_1 := \frac{2-n + \sqrt{(n-2)^2 + 4 \lambda_1}}{2},
\end{equation}
here $\lambda_1$ is the first non-zero eigenvalue of the Laplacian $\Delta_{g^N}$ and $a(y)$ is a corresponding eigenfunction, i.e. $\Delta_{g^N} a(y) = - \lambda_1 a(y)$, 
$u_1 \in L^{2}_{\beta^\prime}$ for some $\beta^\prime< 2-n$, and $u_2 \in L^{2}_{\delta^\prime}$ for some $\delta^\prime > \nu_1$. Note that $B_1 \geq 0$, since $v$ is a nonnegative function on $M \setminus \{o\}$.

Moreover, because $\Delta_g v = 0$ on $B_{r_0}(o)$, we have:
\begin{equation*}
\Delta_g u_2 = - \Delta_g(B_1 + a(y)r^{\nu_1}) = - (\Delta_g - \Delta_{\g})(a(y) r^{\nu_1}) = O(r^{\nu_1 - 2 + \alpha}) \in L^{p}_{\delta - 2}(B_{r_0}(o)),
\end{equation*}
for all $p>1$ and $\nu_1 < \delta < \nu_1 + \alpha$ by \eqref{eqn-polynomial-weighted-Sobolev-condition}. Then by applying the weighted elliptic estimate in Proposition \ref{prop-weight-elliptic-estimate}  (with $p = 2$), we obtain $u_2 \in W^{2, 2}_{\delta}(B_{r_0}(o))$ for $\nu_1 < \delta < \nu_1 + \alpha$. The weighted Sobolev embedding further implies $u_2 \in L^{\frac{2n}{n-2}}_{\delta}(B_{r_0}(o))$ for $\nu_1 < \delta < \nu_1 + \alpha$. Applying Proposition \ref{prop-weight-elliptic-estimate} (with $p=\frac{2n}{n-2}$) again gives $u_2 \in W^{2, \frac{2n}{n-2}}_\delta (B_{r_0}(o))$. By repeating this process, we can obtain that $u_2 \in W^{2, p}_\delta (B_{r_0}(o))$ for all $p>1$ and $\nu_1 < \delta < \nu_1 + \alpha$. Then weighted Sobolev embedding implies 
\begin{equation}
|\nabla^i u_2| = o(r^{\delta - i}), \ \ \text{as} \ \ r \to 0,
\end{equation}
for $\nu_1 < \delta < \nu_1 + \alpha$ and $i = 0, 1$.

Similarly, we can obtain that $u_1$ has the asymptotic at infinity as in (\ref{eqn-asymptotic-u1}).

Thus, 
\begin{equation}
u = 1 + v = 
\begin{cases}
 1 + \frac{A}{\rho^{n-2}} + u_1 \ \ 
 & \text{as} \ \ \rho \to \infty, \\
 B + a(y) r^{\nu_1} + u_2 \ \ 
 & \text{as} \ \ r \to 0,
\end{cases}
\end{equation}
where $u_1$ and $u_2$ have asymptotic as in (\ref{eqn-asymptotic-u1}) and (\ref{eqn-asymptotic-u2}) respectively, and $B= 1 + B_1 \geq 1$. This completes the proof of the lemma.  
\end{proof}

\begin{lem}\label{lem-harmonic-coordinate}
Let $(M^n, g, o)$ be a AF manifold with a single conical singularity at $o$, and $\{x_i\}^{n}_{i=1}$ is the Euclidean coordinate on $M_\infty$ given by the diffeomorphism in Definition \ref{defn-AF-conic-mfld}. There exist harmonic functions $y_i$, $1 \leq i \leq n$, on $M$ admitting the following asymptotic:
\begin{equation}\label{eqn-asymptotic-of-harmonic-coordinate}
    y_{i}=\begin{cases}
        c+r^{\frac{2-n+\sqrt{(n-2)^2+4\lambda_1}}{2}}\phi_{i}+O(r^{\nu'}), & \mbox{ near the conical point } o,\\
        x_{i}-u_{i}, & \mbox{ near the AF infinity},
    \end{cases}
\end{equation}
where $\lambda_{1}$ is the first nonzero eigenvalue of the Laplacian $\Delta_{g^N}$ on the cross section $(N, g^N)$ and $\phi_i$ is a corresponding eigenfunction, i.e. $\Delta_{g^N} \phi_i = - \lambda_1 \phi_i $, and 
\begin{equation*}
\nu^\prime > \frac{2-n + \sqrt{(n-2)^2 + 4 \lambda_1}}{2}>0,
\end{equation*}
and $c\in\mathbb{R}$ is a constant. Moreover,
$u_i$ satisfy
\begin{equation}\label{eqn-asymptotic-of-harmonic-coordinate-infinity}
|u_i| + \rho|\partial u_i| + \rho^2|\partial^2 u_i| = o(\rho^{1-\tau^\prime}), \ \ \text{as} \ \ \rho \to +\infty,
\end{equation}
with $\tau^\prime = \tau - \epsilon > \frac{2-n}{2}$ for sufficiently small $\epsilon >0$.
\end{lem}
\begin{proof}
Let $\chi$ be a cut-off function supported in $M_\infty$ and $\chi = 1$ outside of a compact subset of $M_\infty$. Then 
\begin{equation*}
\Delta (\chi x_i) = 
       \begin{cases}
       O(\rho^{ - \tau -1}), & \text{as} \ \  \rho \to +\infty, \\
       0, & \text{near the conical point} \ \ o,
       \end{cases}
\end{equation*}
and so by \eqref{eqn-polynomial-weighted-Sobolev-condition}
\begin{equation*}
\Delta(\chi x_i) \in L^{2}_{\delta-2, \beta-2}(M), \ \ \text{for} \ \ 2-n < \delta <0, \, \beta > -\tau +1 > 2-n.
\end{equation*}
Therefore, Proposition \ref{prop-surjectivity} implies that there exists $u_i \in L^2_{\delta, \beta}(M)$ for $2-n < \delta <0, \, \beta > -\tau +1 > 2-n$ such that
\begin{equation*}
\Delta u_i = \Delta (\chi x_i).
\end{equation*}
By letting $y_i = \chi x_i - u_i$, we obtain harmonic functions $y_i$, for $1 \leq i \leq n$.

Then note that near the conical point $o$, $y_i = -u_i$, and similarly as in the proof of Lemma 5.2 in \cite{DSW-spin-23}, applying Proposition \ref{prop-asmptotic-order} gives the asymptotic of $y_i$ near the conical point as in (\ref{eqn-asymptotic-of-harmonic-coordinate}). In addition, similarly as in the proof of Lemma 5.1 in \cite{DSW-spin-23},  applying a Nash-Moser iteration argument gives the asymptotic of $u_i$ near infinity as in (\ref{eqn-asymptotic-of-harmonic-coordinate-infinity}).
\end{proof}

Now we are ready to prove the rigidity result in Theorem \ref{thm-main} as following.
\begin{thm}\label{thm-rigidity}
Let $(M^n, g, o)$ be an AF manifold with a single conical singularity at $o$. If the scalar curvature $\Sc_g \geq 0$ and the ADM mass $m(g) = 0$, then $(M^n, g)$ is isometric to the Euclidean space $(\R^n, g_{\R^n})$.
\end{thm}
\begin{proof}
First we show that if $\Sc_g \geq 0$ and $m(g) =  0$ then $g$ is Ricci flat. 
For that, we take a variation $g_{t}=g+th$ of the metric $g$, where $h$ is an arbitrary compactly supported $2$-tensor on $\mathring{M}$. Note that for each $t$ such that $|t|$ is sufficiently small, $g_t$ is still an AF manifold with a single conical singularity at $o$. Then we consider the following equation:
\begin{equation}\label{eqn-rigidity-conformal-equation}
   \begin{cases}
    -\Delta_{g_{t}}u+\frac{n-2}{4(n-1)} \left(\Sc_{g_{t}} -\Sc_{g}\right)u =0,  & \text{on} \ \  \mathring{M},\\
    u \to  1, & \mbox{ as } x\to \infty.
    \end{cases}
\end{equation}
For each $t$ such that $|t|$ is sufficiently small, by Lemma \ref{Lem-conformal-equation}, the equation (\ref{eqn-rigidity-conformal-equation}) has a positive solution $u_{t}$, since $h$ is compactly supported and so is $\Sc_{g_{t}}-\Sc_{g}$, and $\left\|\left( \Sc_{g_t} - \Sc_g \right)\right\|_{L^{\frac{n}{2}}(M)}$ can be arbitrary small as $|t|$ sufficiently small. Therefore, for each $t$ such that $|t|$ is sufficiently small, by the asymptotic of $u_t$ as in (\ref{eqn-asymptotic-u-rigidity}), $\tilde{g}_{t} := u^{\frac{4}{n-2}}_{t}g_{t}$ is a AF metric with a conically singular point at $o$ as in Definition \ref{defn-AF-conic-mfld}. Because $u_t$ is a solution of (\ref{eqn-rigidity-conformal-equation}), we have 
\begin{equation*}
\Sc_{\tilde{g}_{t}} = \frac{4(n-1)}{n-2} (u_t)^{-\frac{n+2}{n-2}} \left( -\Delta_{g_t} u_t + \frac{n-2}{4(n-1)} \Sc_{g_t} u_t \right) = (u_t)^{-\frac{4}{n-2}} \Sc_g \geq 0.
\end{equation*}
Thus, Theorem \ref{thm-nonnegativity-of-mass} implies that the mass $m(\tilde{g}_t) \geq 0$, and so $t=0$ is a interior minimum point of the function $m(\tilde{g}_t)$, since $u_0 \equiv 1, g_0 = g$ and $m(g)=0$.
As a result, 
\[\left.0=\frac{d m(\tilde{g}_{t})}{dt}\right|_{t=0}=\int_{M}h^{jk}\Ric_{jk} d\operatorname{vol}_{g},\]
which implies that $g$ is Ricci flat, since $h$ is arbitrary.

Now we prove that the cross section $(N, g^N)$ of the conical neighborhood of the singular point is the standard sphere. The idea is to compute the mass, $m(g)$, using the harmonic coordinate, $y_i$, obtained in Lemma \ref{lem-harmonic-coordinate}, with the help of Ricci-flatness of $g$ that we have already obtained. As a consequence, $m(g) = 0$ then implies that $dy_i$ are parallel 1-forms, and this then enables us to apply Obata's rigidity theorem for $(N, g^N)$.

Note that for an asymptotically conical metric, $g$, as in Definition \ref{defn-conic-mfld}, its Ricci tensor, $\operatorname{Ric}_{g}$, satisfies
\begin{equation*}
\operatorname{Ric}_g (X, Y) = \operatorname{Ric}_{g^N}(X, Y) - (n-2)g^{N}(X, Y) + O(r^{\alpha}), \ \ \text{as} \ \ r \to 0,
\end{equation*}
for all $X, Y$ tangent to the cross section $N$. Recall that the decay order $\alpha > 0$. Thus, $\Ric_{g} = 0$, implies $\operatorname{Ric}_{g^{N}}=(n-2)g^{N}$ on the cross section. As a result, the Lichnerowicz eigenvalue estimate implies that the first nonzero eigenvalue of $\Delta_{g^N}$, $\lambda_1 \geq n-1$.

Let $y_i, 1 \leq i \leq n,$ be harmonic functions obtained in Lemma \ref{lem-harmonic-coordinate}. By the asymptotic of $y_i$ at infinity as in (\ref{eqn-asymptotic-of-harmonic-coordinate-infinity}), the argument as in the proof of Theorem 4.4 in \cite{Bartnik-CPAM} implies that we can compute the mass, $m(g)$, using $y_i$ as coordinate on $M_\infty$. In addition, by Bochner formula, recall that $g$ is Ricci flat, we have
\begin{eqnarray}
\Delta_{g}\left(\frac{1}{2}|dy_{i}|^2\right)=|\nabla dy_{i}|^2.    
\end{eqnarray}
Then by integrating by part, we have
\begin{eqnarray*}
\omega_n m(g)&=&\lim_{\rho\to \infty}\left[\sum_{i=1}^{n}\int_{\partial B_{\rho}}\partial_{\nu_{\rho}}\left(\frac{1}{2}|dy_{i}|^2\right)\right]\\
&=&\lim_{r\to 0}\left[\sum_{i=1}^{n}\int_{\partial B_{r}}\partial_{\nu_{r}}\left(\frac{1}{2}|dy_{i}|^2\right)\right]+\lim_{r\to0,\rho\to\infty}\left[\sum_{i=1}^{n}\int_{A_{r,\rho}}\left(|\nabla dy_{i}|^2\right)\right]\\
&=&\sum_{i=1}^{n}\int_{M}\left(|\nabla dy_{i}|^2\right)    
\end{eqnarray*}
where $\nu_{\rho}$ and $\nu_{r}$ are the outer normal vector of $\partial B_{\rho}$ and $\partial B_{r}$ respectively, and $\omega_n$ is the volume of the unit sphere in $\R^n$. Here in the last step, we use that the asymptotic behavior of $y_i$ near the conical point as in \eqref{eqn-asymptotic-of-harmonic-coordinate} and the fact $\lambda_1 \geq n-1$.
As a result, $m(g)=0$ implies that $\nabla dy_{i}=0$, i.e. $dy_{i}$ is parallel. Then $|dy^{i}|=1$, since $du_{i}=O(\rho^{-\tau})$. Thus, by considering the asymptotic behavior of $dy_i$ near the conical point, the exponent in $(\ref{eqn-asymptotic-of-harmonic-coordinate})$:
\begin{equation}\label{eqn-exponent}
\frac{2-n+\sqrt{(n-2)^2+4\lambda_{1}}}{2}=1,
\end{equation}
because otherwise $|dy_{i}|$ will tend to either zero or infinity as $r \to 0$, i.e. as approaching to the conical point. 
Therefore, by solving (\ref{eqn-exponent}), we obtain that $\lambda_{1}=n-1$ is an eigenvalue of $\Delta_{g^N}$.
Then we apply the Obata's rigidity theorem to conclude that $(N, g^N)$ must be the standard sphere, since we have shown that $Ric_{g^N} = (n-2)g^N$. As a consequence, the conically singular point $o$ is actually a manifold point, and our metric $g$ is continuous and has $W^{1, n}$-regularity near the conical singularity point, and so it satisfies the regularity assumption in \cite{JSZ2022}. By Theorem 1.1 in \cite{JSZ2022}, we conclude that $(M^{n},g)\cong(\mathbb{R}^{n},g_{\mathbb{R}^{n}})$.
\end{proof}
\begin{rmrk}
{\rm
    In \cite{DSW-spin-23}, for the proof of rigidity, after we get the manifold is Ricci flat, then we can construct the harmonic coordinate to conclude that the manifold is the Euclidean space as we do in the above.
}
\end{rmrk}


\bibliographystyle{plain}
\bibliography{DSW_PMT.bib}

\end{document}